\documentclass[review]{elsarticle}

\usepackage{amssymb,amsmath,mathbbol}

\newtheorem{theorem}{Theorem}[section]
\newtheorem{lemma}[theorem]{Lemma}
\newtheorem{proposition}[theorem]{Proposition}
\newtheorem{corollary}[theorem]{Corollary}
\newdefinition{definition}[theorem]{Definition}
\newdefinition{example}[theorem]{Example}
\newdefinition{remark}[theorem]{Remark}
\newproof{proof}{Proof}
\makeatletter
\def\ps@pprintTitle{%
 \let\@oddhead\@empty
 \let\@evenhead\@empty
 \def\@oddfoot{}%
 \let\@evenfoot\@oddfoot}
\makeatother

\begin{document}

\begin{frontmatter}

\title{Generalized comonotonicity and new axiomatizations of Sugeno integrals on bounded distributive lattice \footnote{Preprint of an article published by Elsevier in the International Journal of Approximate Reasoning 81 (2017), 183-192. It is available online at: www.sciencedirect.com/science/article/pii/S0888613X16302481}}

\author[up]{Radom\'ir Hala\v{s}}
\ead{radomir.halas@upol.cz}

\author[up,stu]{Radko Mesiar}
\ead{radko.mesiar@stuba.sk}

\author[up,sav]{Jozef P\'ocs}
\ead{pocs@saske.sk}

\address[up]{Department of Algebra and Geometry, Faculty of Science, Palack\'y University Olomouc, 17. listopadu 12, 771 46 Olomouc, Czech Republic}
\address[stu]{Department of Mathematics and Descriptive Geometry, Faculty of Civil Engineering, Slovak University of Technology in Bratislava, Radlinsk\'eho 11, 810 05 Bratislava 1, Slovakia}
\address[sav]{Mathematical Institute, Slovak Academy of Sciences, Gre\v s\'akova 6, 040 01 Ko\v sice, Slovakia}

\begin{abstract}

Two new generalizations of the relation of comonotonicity of lattice-valued
vectors are introduced and discussed. These new relations coincide on
distributive lattices and they share several properties with the 
comonotonicity for the real-valued
vectors (which need not hold for $L$-valued vectors
comonotonicity, in general). Based on these newly introduced generalized types
of comonotonicity of $L$-valued vectors, several new axiomatizations of $L$-valued
Sugeno integrals are introduced. One of them brings a substantial decrease of computational complexity when checking an aggregation function to be a Sugeno integral.
\end{abstract}


\end{frontmatter}

\section{Introduction}\label{sec:1}

Sugeno \cite{Sug74} has introduced an integral (called $F$-integral, or fuzzy integral in the original source) for fusion of information obtained in a fuzzy set characterized by its membership function. This integral was acting on $[0,1]$ due to the range of membership functions of fuzzy sets, and it was built by means of the basic fuzzy connectives $\min$ and $\max$, following the introduction of fuzzy sets due to Zadeh \cite{Zadeh1965}. Recall that the Sugeno integral is formally introduced as the Lebesgue and Choquet integrals, replacing the standard arithmetic operations $+$ and $\cdot$ on the real unit interval $[0,1]$ by the lattice operations max and min on the bounded chain $[0,1]$. Similarly as the fuzzy sets with $[0,1]$-valued membership functions were generalized by Goguen \cite{Goguen1967} into $L$-valued fuzzy sets, where $L$ stands for a bounded distributive lattice, also the Sugeno integral can be generalized to act on $L$. Formally, even non-distributive lattices $L$ could be considered, however, then some ambiguities may occur. For example, the values of considered functionals with fixed values in boolean vectors (formally, the underlying fuzzy measure or capacity) need not be unique when two forms of formulas for the Sugeno integral, see Definition \ref{Sug_def}, are considered. Among several deep studies of $L$-valued Sugeno integrals we recall \cite{HMP,Marichal} and especially \cite{Couceiro}, where several equivalent axiomatic characterizations appear. Note that papers \cite{HMP} and \cite{Marichal} study the Sugeno integrals from the different point of views as it is intended in this paper. In a recent paper \cite{HMP} a new characterizing property of Sugeno integrals, based on the preservation of certain equivalence relations (the so-called compatibility), has been presented. In \cite{Marichal} it was shown that Sugeno integrals form a subclass of weighted lattice polynomial functions, which can be characterized by an important median based decomposition formula. 

Let us mention that various types of integrals have many applications within the theory of aggregation functions. To illustrate that the study of Sugeno integrals from various points of view is still very active area of aggregation functions, we refer the reader to recent papers \cite{IJAR1,IJAR4,HP1,HP2,IJAR3,IJAR2}. More specifically, 
in \cite{HP1,HP2} a new promising approach via so-called clone theory to a study of aggregation functions on bounded lattices has been started. We expect to apply these results for better understanding of properties of Sugeno integrals on distributive lattices.  
Moreover, following the spirit of papers \cite{IJAR1,IJAR3,IJAR2} dealing with certain non-additive measures, we expect to introduce Sugeno integrals on non-distributive lattices.

The axiomatic approach to Sugeno integrals was studied in several papers,
including \cite{MAR,nove1,nove2} and it is based on the notion of comonotonicity of real
functions (of real vectors in the case of discrete Sugeno integrals). However,
in the case of lattice-valued vectors, the relation of comonotonicity has some
undesirable properties. For example, there can exist a vector $\mathbf x$ such that
it is not comonotone with itself (i.e., $\mathbf x$ and $\mathbf x$ are not comonotone), neither
it is comonotone with an arbitrary constant vector $\mathbf c$, i.e., $\mathbf x$ and $\mathbf c$ are not comonotone. 
Thus a generalization of the comonotonicity relation for 
$L$-valued vectors avoiding the above mentioned defects is a challenging problem.

The aim of this
contribution is the introduction of generalized comonotonicity relations with
properties similar to the real-valued vectors comonotonicity and a subsequent
development of the Sugeno integral theory on bounded distributive lattices as
an important tool in $L$-valued information fusion. We introduce two new types of comonotonicity, named generalized comonotonicity  and dual generalized comonotonicity here, and we apply them to get new axiomatic characterizations of $L$-valued Sugeno integrals.

The organization of the paper is as follows: 
In Section \ref{sec:2} we recall the basic definitions and notions concerning the discrete Sugeno integrals. The notions of generalized comonotonicity and dual generalized comonotonicity are introduced and studied in Section \ref{sec:3}. New axiomatic characterizations of $L$-valued Sugeno integrals are discussed in Section \ref{sec:4}.

\section{Preliminaries}\label{sec:2}

Let $L$ be a bounded distributive lattice. Through the paper we denote by $0$ and $1$ the bottom and
the top element of $L$, respectively. 
Recall that a function $f\colon L^n\to L$, $n\geq 1$ being an integer, is called an aggregation function (on $L$) whenever it is monotone and satisfies two boundary conditions, $f(0,\dots,0) = 0$ and $f(1,\dots,1) = 1$.

For any integer $n\geq 1$ we set $[n]=\{1,\dots,n\}$. 
Recall that given an $L$-valued capacity $m\colon 2^{[n]}\to L$, i.e., a set function satisfying $m(X)\leq m(Y)$ for $X\subseteq Y\subseteq [n]$ and $m(\emptyset)=0$, $m([n])=1$, there are two equivalent expressions for the Sugeno integral.

\begin{definition}\label{Sug_def}
Let $n\geq 1$ be a positive integer and $m\colon 2^{[n]}\to L$ be an $L$-valued capacity. The Sugeno integral with respect to the capacity $m$ is defined by
\begin{equation}\label{Sugeno_def}
\mathsf{Su}_m(\mathbf{x})=\bigvee_{I\subseteq [n]} \big( m(I)\wedge \bigwedge_{i\in I}x_i\big)= \bigwedge_{I\subseteq [n]}\big( m([n]\smallsetminus I)\vee \bigvee_{i\in I}x_i\big),
\end{equation}
where $\mathbf{x}=(x_1,\dots,x_n)\in L^n$.
\end{definition}

Observe, that for any $L$-valued capacity $m$, the Sugeno integral $\mathsf{Su}_m$ is an aggregation function on $L$.

Two vectors $\mathbf{x}=(x_1,\dots,x_n)\in L^n$ and $\mathbf{y}=(y_1,\dots,y_n)\in L^n$ are said to be \textit{comonotone} if $x_i\leq x_j$ and $y_i\leq y_j$ or $x_i\geq x_j$ and $y_i\geq y_j$ for all pairs $i,j\in\{1,\dots,n\}$. 
Equivalently, $\mathbf{x},\mathbf{y}\in L^n$ are comonotone if and only if there is a permutation $\sigma$ of the set $[n]$ such that $x_{\sigma(1)}\leq\dots\leq x_{\sigma(i)}\leq\dots \leq x_{\sigma(n)}$ and $y_{\sigma(1)}\leq\dots\leq y_{\sigma(i)}\leq\dots \leq y_{\sigma(n)}$.

For $c\in L$, denote by $\mathbf{c}=(c,\dots,c)$ the constant vector. A function $f\colon L^n \to L$ is said to be
\begin{itemize}
\item[--] \textit{inf-homogeneous} if, for every $\mathbf{x}\in L^n$ and every $c\in L$, $f$ satisfies $f(\mathbf{c}\wedge \mathbf{x})= c\wedge f(\mathbf{x})$
\item[--] \textit{sup-homogeneous} if, for every $\mathbf{x}\in L^n$ and every $c\in L$, $f$ satisfies $f(\mathbf{c}\vee \mathbf{x})= c\vee f(\mathbf{x})$
\item[--] \textit{comonotone supremal} if $f(\mathbf{x}\vee \mathbf{y})=f(\mathbf{x})\vee f(\mathbf{y})$ for every pair of comonotone vectors $\mathbf{x},\mathbf{y}\in L^n$
\item[--] \textit{comonotone infimal} if $f(\mathbf{x}\wedge \mathbf{y})=f(\mathbf{x})\wedge f(\mathbf{y})$ for every pair of comonotone vectors $\mathbf{x},\mathbf{y}\in L^n$.
\end{itemize}
Let us note that if $L$ is a chain (e.g., the real line), inf(sup)-homogeneity and comonotone supremality (infimality) are commonly referred to as min(max)-homogeneity and comonotone maxitivity (minitivity) respectively.

We conclude this section with recalling the following well-known characterization of the discrete Sugeno integrals on bounded chains, 
see e.g. \cite{Couceiro,Grabisch et al 2009}.

\begin{proposition}\label{prop1}
Let $L$ be a bounded chain and $f\colon L^n\to L$, $n\geq 1$, be an aggregation function. The following conditions are equivalent: 
\begin{enumerate}
\item[\rm(i)] $f$ is a discrete Sugeno integral.
\item[\rm(ii)] $f$ is comonotone maxitive and min-homogeneous.
\item[\rm(iii)] $f$ is comonotone minitive and max-homogeneous.
\end{enumerate}
\end{proposition}

Let us note that the crucial step in the proof relies on finding an appropriate permutation $\sigma$ of the set $[n]$, such that $x_{\sigma(1)}\leq\dots\leq x_{\sigma(n)}$.
Obviously, this is, in general, no longer possible provided $L$ contains incomparable elements.

\section{Generalized and dually generalized comonotonicity}\label{sec:3}

In this section we introduce the notions of generalized and dually generalized comonotonicity and study their basic properties.

\subsection{Definition and intuition}\label{sec:31}

The notion of comonotonicity of vectors can be introduced for any poset, and, in particular, for any bounded distributive lattice. 
However, then some genuine properties of comonotonicity of real vectors are lost. For example, for real valued vectors (functions), for any vector $\mathbf{x}$
and any constant vector $\mathbf{c}$ the couple $\mathbf{x},\mathbf{c}$ is comonotone. Similarly, the couple $\mathbf{x},\mathbf{x}$ is comonotone for any real valued vector, but not for $L$-valued vectors once $L$ is not a chain. Of course, these problems are caused by a possible incomparability of some elements of $L$. To eliminate the above mentioned
defects, we introduce a new concept of generalized comonotonicity and its dual counterpart.

\begin{definition}\label{def_gcom}
Let $L$ be a lattice. Given two $n$-ary vectors  $\mathbf{x}$ and $\mathbf{y}$, we call them \textit{generalized comonotone} (g-comonotone, for short) if for every pair $i,j\in \{1,\dots,n\}$ we have 
\begin{equation}\label{gcom_def}
(x_i\vee y_i)\wedge(x_j\vee y_j) = (x_i\wedge x_j)\vee (y_i\wedge y_j).
\end{equation}

Two $n$-ary vectors  $\mathbf{x}$ and $\mathbf{y}$ are called \textit{dually generalized comonotone} if for every pair $i,j\in \{1,\dots,n\}$
we have
\begin{equation}\label{dgcom_def}
(x_i\wedge y_i)\vee (x_j\wedge y_j)=(x_i\vee x_j)\wedge (y_i\vee y_j).
\end{equation}
\end{definition}

We call an $n$-ary aggregation function $f$ on $L$ \textit{g-comonotone supremal}, if 
$$f(\mathbf{x}\vee \mathbf{y})=f(\mathbf{x})\vee f(\mathbf{y}),$$
 and dually, $f$ is said to be \textit{g-comonotone infimal}, if 
$$ f(\mathbf{x}\wedge \mathbf{y})=f(\mathbf{x})\wedge f(\mathbf{y}) $$
for any pair of generalized comonotone vectors $\mathbf{x}, \mathbf{y}\in L^n$. The similar notions can be applied in the case of dually generalized comonotone vectors.

Let us remark that the identities \eqref{gcom_def} and \eqref{dgcom_def} are already known in the literature as so-called interchange identities. Interchange identity have its origin in category theory where it is related to a characterization of natural transformations of functors, we refer the reader to a classic book \cite{Mac}.
Not going into details, similar identities relating the Lie and the Jordan products are also deeply studied in computer algebra.

Formally, let $\bullet$ and $\circ$ be two binary operations on a set. Then the following identity 
$$ (a\circ b)\bullet (c\circ d)= (a\bullet c)\circ( b\bullet d) $$
is called the interchange identity (compare also the commuting of aggregation functions discussed in
\cite{SAM}).  

It can be easily seen that putting $\circ=\vee$ and $\bullet=\wedge$ we obtain the identity \eqref{gcom_def}, and similarly, $\circ=\wedge$ and $\bullet=\vee$ yields the dual identity \eqref{dgcom_def}.

Regarding $\vee$ and $\wedge$ as vertical and horizontal compositions respectively, the identity \eqref{gcom_def} expresses the equivalence of two decompositions of a $2\times 2$ array:
$$
(x_i\vee y_i)\wedge(x_j\vee y_j)\; \equiv \;
\left(\begin{array}{c|c}  
  x_i & x_j \\  
  y_i & y_j  
\end{array}\right) 
=
\left(\begin{array}{cc}  
  x_i & x_j \\ \hline 
  y_i & y_j  
\end{array}\right)
\; \equiv \;
(x_i\wedge x_j)\vee (y_i\wedge y_j). 
$$

Hence, two vectors $\mathbf{x},\mathbf{y}\in L^n$ are g-comonotone if and only if for any choice $i,j$ of indexes the above interchange identity is fulfilled.

To give more intuition concerning g-comonotonicity, Figure \ref{fig3} schematically illustrates configuration in a lattice when the inputs $x_i,x_j,y_i,y_j$ are pairwise incomparable elements.

\begin{figure}
\begin{center}
\includegraphics[scale=1.0]{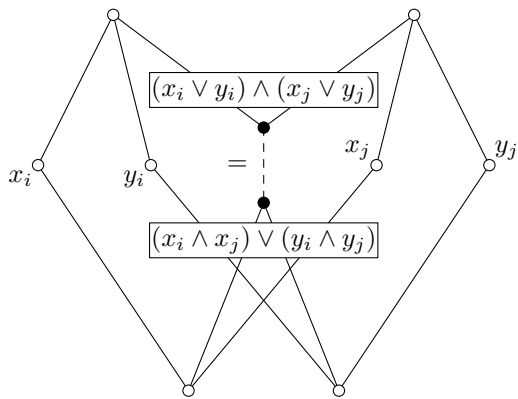}
\end{center}
\caption{Lattice configuration.}
\label{fig3}
\end{figure}

\subsection{Properties}\label{sec:32}

We start an investigation of the introduced notions with the following simple lemma, relating comparability and classical comonotonicity with g-comonocity and dual g-comonocity respectively.

\begin{lemma}\label{lem31}
Let $L$ be a lattice. If $\mathbf{x},\mathbf{y}\in L^n$ are comonotone or comparable, then they are g-comonotone as well as dually g-comonotone.
\end{lemma}

\begin{proof}
Let $\mathbf{x},\mathbf{y}\in L^n$ be comonotone vectors and $i,j\in\{1,\dots,n\}$ be any two indexes. Without loss of generality, assume that $x_i\leq x_j$ and $y_i\leq y_j$. 
Then it is easily seen that 
$$ (x_i\vee y_i)\wedge (x_j\vee y_j)= (x_i\vee y_i) = (x_i\wedge x_j) \vee (y_i\wedge y_j), $$
proving that $\mathbf{x}$ and $\mathbf{y}$ are g-comonotone. Similarly, one can show that $\mathbf{x}$ and $\mathbf{y}$ are also dually g-comonotone.

Further, let $\mathbf{x}$ and $\mathbf{y}$ be two comparable vectors. Assume that $\mathbf{x}\leq\mathbf{y}$. Then for any two indexes $i,j\in\{1,\dots,n\}$ we have $x_i\leq y_i$ and $x_j\leq y_j$. From this we obtain
$$ (x_i\vee y_i)\wedge (x_j\vee y_j)= y_i\wedge y_j= (x_i\wedge x_j) \vee (y_i\wedge y_j).$$

The dual g-comonotonicity can be proved analogously.
\qed
\end{proof}

The previous lemma shows that the notions of g-comonotonicity and dual g-comonotonicity generalize that of comonotonicity and  comparability. 
Moreover, for a constant vector $\mathbf{c}$, a vector $\mathbf{x}\in L^n$ is comonotone with $\mathbf{c}$ only if the set 
$\{x_1,\dots,x_n\}$ forms a chain in $L$. Similarly, if $c\in L\setminus\{0,1\}$, there are vectors which are not comparable with $\mathbf{c}$. However, for any vector $\mathbf{x}\in L^n$, $\mathbf{x}$ and $\mathbf{c}$ are g-comonotone (dually g-comonotone). 
Indeed, substituting $y_i=y_j=c$ in (\ref{gcom_def}), then applying distributivity of $L$, we obtain the equality
$$(x_i\vee c)\wedge (x_j\vee c)=(x_i\wedge x_j)\vee c=(x_i\wedge x_j)\vee (c\wedge c)$$
for any $i,j\in \{1,\dots,n\}$. Similarly, the dual g-comonotonicity (\ref{dgcom_def}) for such a pair of vectors can be verified.

\begin{example}\label{ex1}
Consider the product $L=[0,1]\times [0,1]$. For a fixed element $\mathbf{x}=(x_1,x_2)\in [0,1]^2$, consider the sets of all vectors $\mathbf{y}$ such that $\mathbf{x}$ and $\mathbf{y}$ are comonotone, comparable and g-comonotone, respectively. Particularly, we put $A(\mathbf{x})=\{\mathbf{y}\mid \mathbf{x},\mathbf{y}\ \mbox{comonotone}\}$, $B(\mathbf{x})=\{\mathbf{y}\mid \mathbf{x},\mathbf{y}\ \mbox{comparable}\}$ and $C(\mathbf{x})=\{\mathbf{y}\mid \mathbf{x},\mathbf{y}\ \mbox{g-comonotone}\}$. For a point $\mathbf{x}=(x_1,x_2)\in [0,1]^2$ with $x_1>x_2$, these sets are depicted in Figure \ref{fig2}.   
\begin{figure}
\begin{center}
\includegraphics[scale=0.6]{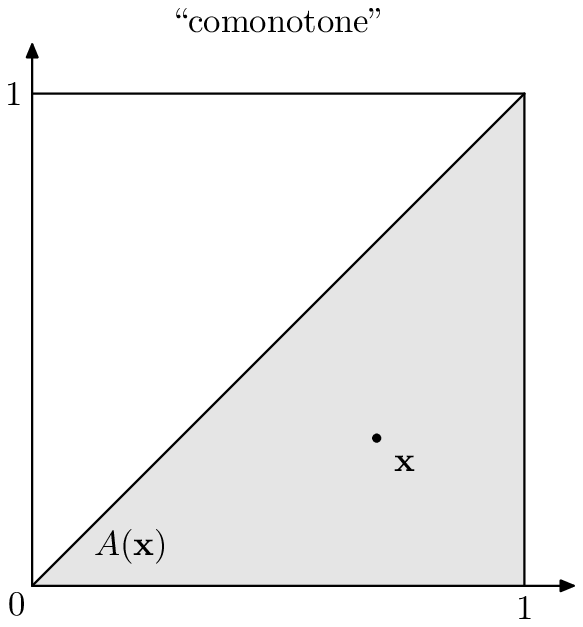}
\includegraphics[scale=0.6]{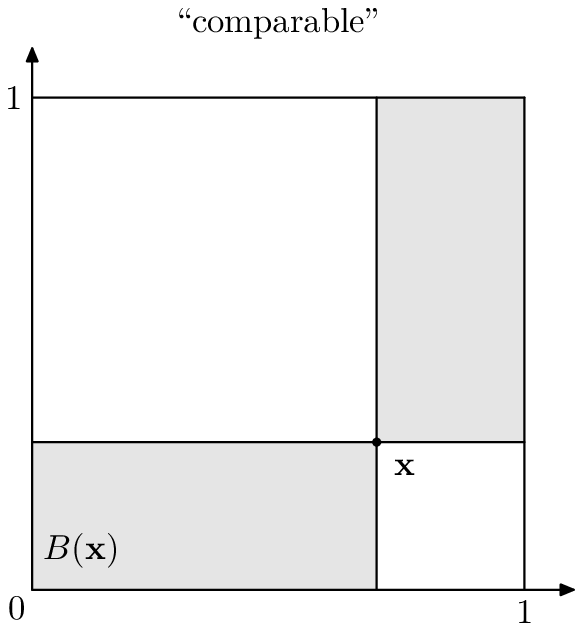}
\includegraphics[scale=0.6]{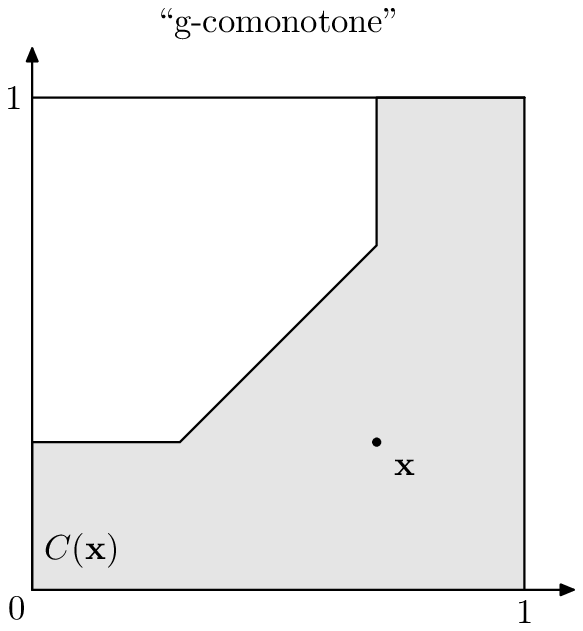}
\end{center}
\caption{The sets $A(\mathbf{x})$, $B(\mathbf{x})$ and $C(\mathbf{x})$ for a given vector $\mathbf{x}\in[0,1]^2$.}
\label{fig2}
\end{figure} 
As the figure indicates, for this particular point $\mathbf{x}$, the set $C(\mathbf{x})$ is the union of the sets $A(\mathbf{x})$ and $B(\mathbf{x})$. 

We show that this is valid in general, i.e., $C(\mathbf{x})=A(\mathbf{x})\cup B(\mathbf{x})$ for all $\mathbf{x}\in [0,1]^2$.
To observe this, let $\mathbf{x}=(x_1,x_2)$ be such that $x_1>x_2$ and $\mathbf{y}=(y_1,y_2)$ be arbitrary. Assume that $\mathbf{y}\notin A(\mathbf{x})\cup B(\mathbf{x})$. Then necessarily $y_1<y_2$ (since $\mathbf{x}$, $\mathbf{y}$ are not comonotone), and $ y_1<x_1$, $y_2>x_2$ as $\mathbf{x}$, $\mathbf{y}$ are incomparable. 

However, from this we obtain 
$$(x_1\vee y_1)\wedge (x_2\vee y_2)=x_1\wedge y_2 > x_2 \vee y_1= (x_1\wedge x_2)\vee (y_1\wedge y_2),$$ i.e., $\mathbf{x}$, $\mathbf{y}$ are not g-comonotone. Hence $C(\mathbf{x})\subseteq A(\mathbf{x})\cup B(\mathbf{x})$ and this inclusion can be also proved for $\mathbf{x}$ satisfying $x_1<x_2$. Note that $A(\mathbf{x})=[0,1]^2$ provided $\mathbf{x}=(x,x)$ for some $x\in [0,1]$. Since the opposite inclusion $C(\mathbf{x})\supseteq A(\mathbf{x})\cup B(\mathbf{x})$ follows from Lemma \ref{lem31}, we obtain $C(\mathbf{x})=A(\mathbf{x})\cup B(\mathbf{x})$ for all $\mathbf{x}\in[0,1]^2$.
\end{example}

Let us remark that the above equality $C(\mathbf{x})=A(\mathbf{x})\cup B(\mathbf{x})$ does not hold in higher dimensions. 

\begin{example}
Consider $L=[0,1]^3$ and the vectors $\mathbf{x}=(0.6,0.3,0.5)$ and $\mathbf{y}=(0.7,0.2,0.9)$. It can be easily seen that they are incomparable as well as they are not comonotone. On the other hand, they fulfill equalities
$$ (x_1\vee y_1)\wedge (x_2\vee y_2)=0.7\ \wedge\ 0.3 =0.3= 0.3\ \vee\ 0.2= (x_1\wedge x_2)\vee (y_1\wedge y_2),$$
$$ (x_1\vee y_1)\wedge (x_3\vee y_3)=0.7\ \wedge\ 0.9 =0.7= 0.5\ \vee\ 0.7= (x_1\wedge x_3)\vee (y_1\wedge y_3),$$
$$ (x_2\vee y_2)\wedge (x_3\vee y_3)=0.3\ \wedge\ 0.9 =0.3= 0.3\ \vee\ 0.2= (x_2\wedge x_3)\vee (y_2\wedge y_3),$$
showing that they are g-comonotone. 
\end{example}

Observe that the generalized comonotonicity in $[0,1]^n$, $n\geq 2$ can be characterized as follows:  
by definition, two vectors $\mathbf{x}$ and $\mathbf{y}$ are g-comonotone if and only if for any $i, j \in\{1,\dots,n\}$ the pairs $(x_i,x_j)$ and $(y_i,y_j)$ satisfy \eqref{gcom_def}. 
However, given fixed $i,j\in\{1,\dots,n\}$, according to Example \ref{ex1} the equation \eqref{gcom_def} is valid if and only if 
$(y_i,y_j)\in A(x_i,x_j)\cup B(x_i,x_j)$, i.e., when the pairs $(x_i,x_j)$ and $(y_i,y_j)$ are comparable or $(x_i,x_j)$ and $(y_i,y_j)$ are comonotone.

Coming back to our example, note that the vectors $(x_1,x_2)$ and $(y_1,y_2)$ are comonotone, the vectors $(x_1,x_3)$ and $(y_1,y_3)$ are comparable, while
the vectors $(x_2,x_3)$ and $(y_2,y_3)$ are comonotone.

For an arbitrary lattice $L$, the notions of g-comonotonicity and dual g-comonotonicity need not be equivalent. However, in what follows we show that this will be the case when considering the distributive lattices.

\begin{theorem}\label{thm1}
On any distributive lattice $L$, generalized comonotonicity is self-dual, i.e., it is equivalent 
to dual generalized comonotonicity.
\end{theorem}

\begin{proof}
Assume that $\mathbf{x},\mathbf{y}\in L^n$ are g-comonotone and $\{i,j\}\subseteq \{1,\dots,n\}$ be a pair of indexes. Then 
$$(x_i\vee y_i)\wedge (x_j\vee y_j)=(x_i\wedge x_j)\vee (y_i\wedge y_j).$$
Applying distributivity of $L$, we obtain 
$$ (x_i\wedge x_j)\vee(y_i\wedge x_j)\vee(x_i\wedge y_j)\vee(y_i\wedge y_j)=(x_i\wedge x_j)\vee (y_i\wedge y_j),$$
which is equivalent to 
$$(x_i\wedge y_j)\vee (y_i\wedge x_j)\leq (x_i\wedge x_j)\vee (y_i\wedge y_j).$$
Note that the equivalence follows from the fact that in any lattice $M$, for $a,b\in M$ we have $a\vee b = b$ if and only if $a\leq b$. In our case $a=(x_i\wedge y_j)\vee (y_i\wedge x_j)$ and $b=(x_i\wedge x_j)\vee (y_i\wedge y_j)$.

Now, applying distributivity once more and the above inequality for the right-hand side of the dual g-comonotonicity (\ref{dgcom_def}) we obtain
$(x_i\vee x_j)\wedge (y_i\vee y_j)=(x_i\wedge y_i)\vee (x_i\wedge y_j)\vee (x_j\wedge y_i)\vee (x_j\wedge y_j)\leq 
(x_i\wedge y_i)\vee (x_i\wedge x_j)\vee (y_i\wedge y_j)\vee (x_j\wedge y_j).$
Clearly, $(x_i\wedge y_i)\vee (x_i\wedge x_j)=x_i\wedge (y_i\vee x_j)\leq x_i$, $(y_i\wedge y_j)\vee (x_j\wedge y_j)=(y_i\vee x_j)\wedge y_j\leq y_j$, hence 
$(x_i\vee x_j)\wedge (y_i\vee y_j)\leq x_i\vee y_j.$

Similarly, $(x_i\wedge y_i)\vee (y_i\wedge y_j)\leq y_i$ and $(x_i\wedge x_j)\vee (x_j\wedge y_j)\leq x_j$, i.e. 
$(x_i\vee x_j)\wedge (y_i\vee y_j)\leq x_j\vee y_i$ and we have 
$$ (x_i\vee x_j)\wedge (y_i\vee y_j)\leq (x_j\vee y_i)\wedge (x_i\vee y_j)$$
or equivalently
$$ (x_i\vee x_j)\wedge (y_i\vee y_j)\wedge (x_j\vee y_i)\wedge (x_i\vee y_j)= (x_i\vee x_j)\wedge (y_i\vee y_j).$$
By distributivity of $L$, this is equivalent to
$$ (x_i\wedge y_i)\vee (x_j\wedge y_j)=(x_i\vee x_j)\wedge (y_i\vee y_j),$$
which shows that $\mathbf{x}$ and $\mathbf{y}$ are dually g-comonotone. 

The converse implication can be done in a similar way by using dual arguments. 
\qed
\end{proof}

Before we prove an important lemma concerning g-comonotonicity, or equivalently dual g-comonotonicity in distributive lattices, we recall the following well-known fact, cf. \cite{G1}.

\begin{remark}
Let $L$ be a distributive lattice and let $(\lambda_{i,0})_{i\in I}$, $(\lambda_{i,1})_{i\in I}$, $I\neq \emptyset$ finite, be two families of elements of $L$. 
Then 
\begin{equation}\label{eq_d}
\bigwedge_{i\in I}(\lambda_{i,0}\vee \lambda_{i,1})=\bigvee_{\varphi\in \{0,1\}^I}\bigwedge_{i\in I} \lambda_{i,\varphi(i)}
\end{equation}
and dually
\begin{equation}\label{eq_d1}
\bigvee_{i\in I}(\lambda_{i,0}\wedge \lambda_{i,1})=\bigwedge_{\varphi\in \{0,1\}^I}\bigvee_{i\in I} \lambda_{i,\varphi(i)},
\end{equation}
where $\{0,1\}^I=\big\{\varphi\mid \varphi\colon I\to\{0,1\}\big\}$ denotes the set of all functions with domain $I$ and values in $\{0,1\}$. 
\end{remark}

\begin{lemma}\label{lem1}
Let $L$ be a distributive lattice and $\mathbf{x},\mathbf{y}\in L^n$ be two $n$-ary vectors. 
Then  $\mathbf{x}$ and $\mathbf{y}$ are g-comonotone if and only if 
\begin{equation}\label{gcom}
\bigwedge_{i\in I} (x_i\vee y_i)=\bigwedge_{i\in I}x_i \ \vee\ \bigwedge_{i\in I}y_i
\end{equation}
for any non-empty subset $I\subseteq \{1,\dots, n\}$.

Similarly, $\mathbf{x}$ and $\mathbf{y}$ are dually g-comonotone if and only if
\begin{equation}\label{dgcom}
\bigvee_{i\in I} (x_i\wedge y_i)=\bigvee_{i\in I}x_i \ \wedge\ \bigvee_{i\in I}y_i
\end{equation}
for any non-empty subset $I\subseteq \{1,\dots, n\}$.
\end{lemma}

\begin{proof}
We prove the first equivalence. The second one can be proved using the dual arguments.

Obviously, \eqref{gcom} applied to a two-element subset $I=\{i,j\}$ yields \eqref{gcom_def}, i.e., two vectors $\mathbf{x}$, $\mathbf{y}$ are g-comonotone, provided they satisfy \eqref{gcom}.

Conversely, assume that $\mathbf{x}$ and $\mathbf{y}$ are g-comonotone. Then evidently \eqref{gcom} holds for $I=\emptyset$ as well as for any one or two-element subset $I\subseteq \{1,\dots,n\}$. Thus, assume further that $I\subseteq \{1,\dots,n\}$ is an arbitrary subset, $\left|I\right|=m$ where $3\leq m\leq n$ and that \eqref{gcom} is valid for any subset $J\subseteq \{1,\dots,n\}$ with $\left|J\right|=m-1$. Then with respect to the induction hypothesis, we obtain
$$ \bigwedge_{i\in I} (x_i\vee y_i)=\bigwedge_{i\in I}\bigwedge_{j\in I\smallsetminus\{i\}}\hspace{-0.3cm}(x_j\vee y_j)=\bigwedge_{i\in I}\big( \bigwedge_{j\in I\smallsetminus\{i\}}\hspace{-0.3cm}x_j\;\; \vee \bigwedge_{j\in I\smallsetminus\{i\}}\hspace{-0.3cm}y_j\big).$$

For $i\in I$ put $\displaystyle \lambda_{i,0}=\hspace{-0.3cm}\bigwedge_{j\in I\smallsetminus\{i\}}\hspace{-0.3cm}x_j$ and $\displaystyle \lambda_{i,1}=\hspace{-0.3cm}\bigwedge_{j\in I\smallsetminus\{i\}}\hspace{-0.3cm}y_j$. According to \eqref{eq_d} we have
$$ \bigwedge_{i\in I} (x_i\vee y_i)=\bigwedge_{i\in I}(\lambda_{i,0}\vee \lambda_{i,1})=\bigvee_{\varphi\in \{0,1\}^I}\bigwedge_{i\in I} \lambda_{i,\varphi(i)}.$$
For the constant functions $\varphi_0,\varphi_1\colon I\to \{0,1\}$ such that $\varphi_0(i)=0$ and $\varphi_1(i)=1$ for all $i\in I$ we have
$$ \bigwedge_{i\in I} \lambda_{i,\varphi_0(i)}=\bigwedge_{i\in I}\bigwedge_{j\in I\smallsetminus\{i\}}\hspace{-0.3cm}x_j=\bigwedge_{i\in I}x_i \quad \mbox{and}\quad 
\bigwedge_{i\in I} \lambda_{i,\varphi_1(i)}=\bigwedge_{i\in I}\bigwedge_{j\in I\smallsetminus\{i\}}\hspace{-0.3cm}y_j=\bigwedge_{i\in I}y_i. $$
If $\varphi\colon I\to \{0,1\}$ is non-constant, then $\varphi^{-1}(0)$ or $\varphi^{-1}(1)$ contains at least two elements since $\left|I\right|\geq 3$. Suppose that $\{i_1,i_2\}\subseteq \varphi^{-1}(0)$. Then 
$$ \bigwedge_{i\in I} \lambda_{i,\varphi(i)}\leq \lambda_{i_1,\varphi(i_1)}\wedge \lambda_{i_2,\varphi(i_2)}=\lambda_{i_1,0}\wedge \lambda_{i_2,0}=\bigwedge_{j\in I\smallsetminus\{i_1\}}\hspace{-0.3cm}x_j\ \wedge\ \bigwedge_{j\in I\smallsetminus\{i_2\}}\hspace{-0.3cm}x_j=\bigwedge_{i\in I}x_i.$$
Similarly, 
$$ \bigwedge_{i\in I} \lambda_{i,\varphi(i)}\leq \bigwedge_{i\in I}y_i,$$ 
provided $\{i_1,i_2\}\subseteq \varphi^{-1}(1)$. Consequently, we obtain 
$$ \bigwedge_{i\in I} (x_i\vee y_i)=\bigvee_{\varphi\in \{0,1\}^I}\bigwedge_{i\in I} \lambda_{i,\varphi(i)}=\bigwedge_{i\in I}x_i\ \vee\ \bigwedge_{i\in I}y_i \ \vee\ \bigvee_{\substack{\varphi\in \{0,1\}^I \\ \varphi\neq\varphi_0,\varphi_1}}\bigwedge_{i\in I} \lambda_{i,\varphi(i)}=\bigwedge_{i\in I}x_i\ \vee\ \bigwedge_{i\in I}y_i,$$
which shows that \eqref{gcom} also holds for the subset $I$.
\qed
\end{proof}

As a consequence of the previous two assertions we obtain the following theorem.

\begin{theorem}\label{thm2}
Let $L$ be a distributive lattice and $\mathbf{x}, \mathbf{y}\in L^n$ be two $n$-ary vectors. The following conditions are equivalent:
\begin{enumerate}
\item[\rm(i)] The vectors $\mathbf{x}$ and $\mathbf{y}$ are g-comonotone.
\item[\rm(ii)] The vectors $\mathbf{x}$ and $\mathbf{y}$ are dually g-comonotone.
\item[\rm(iii)] $\displaystyle \bigwedge_{i\in I} (x_i\vee y_i)=\bigwedge_{i\in I}x_i \ \vee\ \bigwedge_{i\in I}y_i$ holds for each subset $I\subseteq \{1,\dots, n\}$.
\item[\rm(iv)] $\displaystyle \bigvee_{i\in I} (x_i\wedge y_i)=\bigvee_{i\in I}x_i \ \wedge\ \bigvee_{i\in I}y_i$ holds for each subset $I\subseteq \{1,\dots, n\}$.
\end{enumerate}
\end{theorem}

\begin{proof}
According to Theorem \ref{thm1} the conditions (i) and (ii) are equivalent, while (i) if and only if (iii) and (ii) if and only if (iv) follow from Lemma \ref{lem1}. 
\qed
\end{proof}

\section{Alternative characterizations of Sugeno integrals on bounded distributive lattices}\label{sec:4}

The goal of this section is to study some of the characteristic properties of $L$-valued Sugeno integrals, mainly with respect to the defined concept of g-comonotonicity.

\begin{lemma}\label{lem2}
Let $L$ be a bounded distributive lattice and $m\colon 2^{[n]}\to L$ be an $L$-valued capacity. Then the discrete Sugeno integral $\mathsf{Su}_m$ is inf-homogeneous and g-comonotone supremal as well as sup-homogeneous and g-comonotone infimal.
\end{lemma}

\begin{proof}
Let $m\colon 2^{[n]}\to L$ be a capacity given on $[n]$. Recall (see Definition \ref{Sug_def}), that the corresponding Sugeno integral is defined for all $\mathbf{x}=(x_1,\dots,x_n)\in L^n$ by 
$$ \mathsf{Su}_m(\mathbf{x})=\bigvee_{I\subseteq [n]}\big(m(I)\wedge \bigwedge_{i\in I}x_i\big). $$

For an element $c\in L$ and the constant vector $\mathbf{c}=(c,\dots,c)$ we obtain 
$$\mathsf{Su}_m(\mathbf{c}\wedge\mathbf{x})=\bigvee_{I\subseteq [n]}\big(m(I)\wedge \bigwedge_{i\in I}(c\wedge x_i)\big)=\bigvee_{I\subseteq [n]}\big(c \wedge m(I)\wedge \bigwedge_{i\in I}x_i\big).$$  
Since $m(\emptyset)=0$, distributivity of $L$ yields
$$ \mathsf{Su}_m(\mathbf{c}\wedge\mathbf{x})=c \wedge\bigvee_{I\subseteq [n]}\big( m(I)\wedge \bigwedge_{i\in I}x_i\big)= c\wedge \mathsf{Su}_m(\mathbf{x}),$$
i.e., $\mathsf{Su}_m$ is inf-homogeneous.

Further, to show that $\mathsf{Su}_m$ is g-comonotone supremal, let $\mathbf{x},\mathbf{y}\in L^n$ be two generalized comonotone vectors. Then due to (iii) of Theorem \ref{thm2} we have
$$ \mathsf{Su}_m(\mathbf{x}\vee \mathbf{y})=\bigvee_{I\subseteq [n]}\big(m(I)\wedge \bigwedge_{i\in I}(x_i\vee y_i)\big)= \bigvee_{I\subseteq [n]}\Big(m(I)\wedge \big(\bigwedge_{i\in I}x_i\  \vee\ \bigwedge_{i\in I}y_i\big)\Big),$$ 
which is by distributivity equal to
$$ \bigvee_{I\subseteq [n]}\Big( \big(m(I)\wedge \bigwedge_{i\in I}x_i \big) \vee \big(m(I)\wedge \bigwedge_{i\in I}y_i\big) \Big)= 
\bigvee_{I\subseteq [n]}\big(m(I)\wedge \bigwedge_{i\in I}x_i\big) \vee  \bigvee_{I\subseteq [n]}\big(m(I)\wedge \bigwedge_{i\in I}y_i\big).$$
Hence $\mathsf{Su}_m(\mathbf{x}\vee \mathbf{y})=\mathsf{Su}_m(\mathbf{x})\vee \mathsf{Su}_m(\mathbf{y})$, proving that the Sugeno integral $\mathsf{Su}_m$ is g-comonotone supremal.

Using the dual expression for the Sugeno integral, one can show in the same way that $\mathsf{Su}_m$ is also sup-homogeneous and g-comonotone infimal.
\qed
\end{proof}

\begin{remark}\label{rem2}
According to Lemma \ref{lem31}, as comonotone vectors are also g-comonotone, we obtain that any discrete Sugeno integral on a bounded distributive lattice is comonotone supremal and comonotone infimal. Obviously, g-comonotone supremality of a function $f$ represents the stronger condition than comonotone supremality in general, since the equality $f(\mathbf{x}\vee\mathbf{y})=f(\mathbf{x})\vee f(\mathbf{y})$ is required for more pairs of vectors in the former case.
\end{remark}
Before we prove that inf-homogeneity together with comonotone supremality characterize discrete Sugeno integrals, we recall one result from \cite{Couceiro}. 

\begin{proposition}\label{prop2}
An aggregation function $f$ on a bounded distributive lattice $L$ is a discrete Sugeno integral if and only if $f$ is inf-homogeneous and sup-homogeneous.
\end{proposition}

With respect to this result a natural question can be raised whether comonotone supremality implies sup-homogeneity. If this is the case, then the proposed characterization would be just a simple consequence of the result from \cite{Couceiro}. In what follows we will discuss this question. First, we prove the following simple, but important lemma.

\begin{lemma}\label{lem4}
Let $L$ be a bounded lattice.
If $f\colon L^n \to L$ is an inf-homogeneous or sup-homogeneous aggregation function, then $f$ is idempotent. 
\end{lemma}

\begin{proof}
Assume that $f$ is inf-homogeneous. Then  
$$ f(x,\dots,x)= x\wedge f(1,\dots,1)=x\wedge 1=x.$$ 
Dually, if $f$ is sup-homogeneous, we obtain 
$$ f(x,\dots,x)= x\vee f(0,\dots,0)=x\vee 0=x,$$ 
hence in both cases the function $f$ is idempotent. 
\qed
\end{proof}

In what follows we give an example of aggregation function which is comonotone supremal but fails to be sup-homogeneous. 

\begin{example}
Let $L$ be a bounded distributive lattice with at least three elements and $n\geq 1$ an integer. Consider the constant aggregation function $h\colon L^n\to L$ given by $h(0,\dots,0)=0$ and $h(\mathbf{x})=1$ otherwise. It is easily seen that $h$ is a $\vee$-homomorphism, i.e., $h(\mathbf{x}\vee \mathbf{y})=h(\mathbf{x})\vee h(\mathbf{y})$ for all $\mathbf{x},\mathbf{y}\in L^n$, thus it is comonotone supremal. On the other hand, $h$ is not idempotent, therefore it cannot be sup-homogeneous. 
Note that if $L$ is a two element chain, then each aggregation function on $L$ is inf-homogeneous as well as sup-homogeneous. 
Similarly, one can find an aggregation function which is comonotone infimal but fails to be inf-homogeneous.
\end{example}

Let us note that the previous consideration, i.e., the usage of non-surjective $\vee$-homomorphism, can be modified to obtain a relatively rich class of
aggregation functions which are comonotone supremal but not sup-homogeneous. Let $L$ be a bounded distributive lattice, $\left|L\right|\geq 3$. Given an integer $n\geq 1$, consider a function $h\colon L^n\to L$ which is comonotone supremal, e.g., $\mathsf{Su}_m$ represents such function. Further, let $g\colon L\to L$ be a non-surjective $\vee$-homomorphism preserving the bottom and the top element of $L$ respectively. Then the composition $h\circ g\colon L^n\to L$ is also a comonotone supremal aggregation function, which fails to be idempotent (it is not surjective). With respect to Lemma \ref{lem4} the function $h\circ g$ cannot be sup-homogeneous.

The following statement is a simple corollary of Proposition \ref{prop2}.

\begin{corollary}\label{lem5}
Let $L$ be a distributive lattice and $f\colon L^n \to L$ be an inf-homogeneous and g-comonotone supremal aggregation function. Then $f$ is a Sugeno integral. 
\end{corollary}

\begin{proof}
If $f$ is inf-homogeneous, then it is idempotent. Since every constant vector $\mathbf{c}=(c,\dots,c)$, $c\in L$, is g-comonotone with any $\mathbf{x}\in L^n$, from g-comonotone supremality of $f$ it follows that $f(\mathbf{c}\vee \mathbf{x})=f(\mathbf{c})\vee f(\mathbf{x})=c \vee f(\mathbf{x})$, i.e., $f$ is sup-homogeneous. Consequently, Proposition \ref{prop2} yields that $f$ is a Sugeno integral.
\qed
\end{proof}

\begin{remark}\label{rem5}
{\rm Note, that using the dual arguments, it can be similarly shown that if $f$ is a sup-homogeneous and g-comonotone infimal aggregation function, then $f$ is a Sugeno integral.}
\end{remark}

In the next theorem we introduce an important simplification of Proposition \ref{prop2}. Following the notation introduced in \cite{Couceiro2}, we say that $f\colon L^n \to L$ is Boolean inf-homogeneous (resp. Boolean sup-homogeneous) if 
$$ f(\mathbf{c}\wedge \mathbf{x})=c \wedge f(\mathbf{x}) \quad\quad \big(\ f(\mathbf{c}\vee \mathbf{x})= c\vee f(\mathbf{x}) \ \big)$$ 
for all $\mathbf{x}\in \{0,1\}^n\subseteq L^n$ and for all constant vectors $\mathbf{c}=(c,\dots,c)$, $c\in L$.

\begin{lemma}\label{lem4n}
If a function $f\colon L^n\to L$ is Boolean inf-homogeneous and Boolean sup-homogeneous, then it is idempotent.
\end{lemma}

\begin{proof}
For any $c\in L$ we obtain:
$$ c\leq f(\mathbf{0})\vee c=f(\mathbf{0}\vee \mathbf{c})=f(\mathbf{c})= f(\mathbf{1}\wedge \mathbf{c})=f(\mathbf{1})\wedge c\leq c.$$
\end{proof} 

\begin{theorem}\label{thm4}
Let $L$ be a distributive lattice and $f\colon L^n\to L$ be an aggregation function on $L$. The function $f$ is a discrete Sugeno integral on $L$ if and only if it is Boolean 
sup-homogeneous and Boolean inf-homogeneous. 
\end{theorem}

\begin{proof}
If $f$ is a Sugeno integral on $L$, then according to Proposition \ref{prop2}, $f$ is inf and sup-homogeneous. Consequently it is clear that the ordinary inf(sup)-homogeneity implies the Boolean inf(sup)-homogeneity. 

Conversely, assume that $f$ is Boolean inf and sup-homogeneous.
For $I\subseteq [n]$, put $m(I)=f(\mathbf{1}_I)$. Obviously $m\colon 2^{[n]}\to L$ is an $L$-valued capacity on the set $[n]$, since $f$ is non-decreasing and fulfills the boundary conditions.

For all $\mathbf{x}\in L^n$ we show that $\mathsf{Su}_m(\mathbf{x})\leq f(\mathbf{x})$ as well as $f(\mathbf{x})\leq \mathsf{Su}_m(\mathbf{x})$.
Given a subset $I\subseteq [n]$, let $\mathbf{u}_I=(\bigwedge_{i\in I}x_i,\dots,\bigwedge_{i\in I}x_i)$ be the constant vector. Then $\mathbf{u}_I\wedge \mathbf{1}_I\leq \mathbf{x}$. Indeed, 
$(\mathbf{u}_I\wedge \mathbf{1}_I)_i=0 \leq x_i$ if $i\notin I$, while $(\mathbf{u}_I\wedge \mathbf{1}_I)_i=\bigwedge_{j\in I}x_j\leq x_i$ provided $i\in I$.
As $f$ is idempotent (Lemma \ref{lem4}) and Boolean inf-homogeneous, we obtain 
\begin{equation}\label{eqn1} 
\bigwedge_{i\in I} x_i \wedge m(I)=\bigwedge_{i\in I}x_i\wedge f(\mathbf{1}_I)=f( \mathbf{u}_I\wedge \mathbf{1}_I)\leq f(x_1,\dots,x_n).
\end{equation}
Consequently, since \eqref{eqn1} holds for each subset $I\subseteq [n]$, we have
$$ \mathsf{Su}_m(\mathbf{x})=\bigvee_{I\subseteq [n]} \big( m(I) \wedge \bigwedge_{i\in I}x_i \big)\leq f(\mathbf{x}).$$

Conversely, for a subset $I\subseteq [n]$, let $\mathbf{v}_I=(\bigvee_{i\in I}x_i,\dots,\bigvee_{i\in I}x_i)$ be the constant vector. Then $\mathbf{x}\leq \mathbf{v}_I \vee \mathbf{1}_{[n]\smallsetminus I}$ and similarly as in the previous case we obtain
\begin{equation}\label{eqn2}
f(x_1,\dots, x_n)\leq f\big( \mathbf{v}_I \vee \mathbf{1}_{[n]\smallsetminus I}\big)=f(\mathbf{v}_I) \vee f(\mathbf{1}_{[n]\smallsetminus I})
=\displaystyle\bigvee_{i\in I}x_i \vee\ m([n]\smallsetminus I).
\end{equation}

Consequently, from \eqref{eqn2} using the dual expression for a Sugeno integral, we obtain
$$ f(\mathbf{x})\leq \bigwedge_{I\subseteq [n]}\big( m([n]\smallsetminus I) \vee \bigvee_{i\in I}x_i \big)=\mathsf{Su}_m(\mathbf{x}). $$
\qed
\end{proof}

Let us remark that Lemma \ref{lem4n} and Theorem \ref{thm4} were proved for bounded chains in \cite{Couceiro2}.

\begin{lemma}
Let $L$ be a bounded distributive lattice and $f\colon L^n\to L$ be an aggregation function. If $f$ is comonotone supremal, then $f$ is Boolean sup-homogeneous. Dually, if $f$ is comonotone infimal, then it is Boolean inf-homogeneous.
\end{lemma}

\begin{proof}
Let $\mathbf{c}=(c,\dots,c)$, $c\in L$ and $\mathbf{x}=(x_1,\dots,x_n)\in \{0,1\}^n$ be arbitrary vectors. Since the set $\{x_1,\dots,x_n\}\subseteq \{0,1\}$ forms a chain in $L$, it follows that $\mathbf{c}$ and $\mathbf{x}$ are comonotone. Hence the comonotone supremality (comonotone infimality) of $f$ implies that $f$ is Boolean sup-homogeneous (Boolean inf-homogeneous).
\qed
\end{proof}

Applying the previous lemma and Theorem \ref{thm4} we obtain the following corollaries.

\begin{corollary}\label{cor1}
If an aggregation function is comonotone supremal and comonotone infimal, then it is a Sugeno integral.
\end{corollary}

\begin{corollary}\label{cor2}
If an aggregation function is inf-homogeneous and comonotone supremal or sup-homogeneous and comonotone infimal, then it is a Sugeno integral.
\end{corollary}

Summarizing our results we obtain the following seven equivalent axiomatic characterizations of the discrete $L$-valued Sugeno integrals:

\begin{theorem}\label{thm3}
Let $L$ be a bounded distributive lattice and let $f\colon L^n\to L$ be an $n$-ary aggregation function on $L$. The following conditions are equivalent: 
\begin{enumerate}
\item[\rm(i)] $f$ is a Sugeno integral on $L$.
\item[\rm(ii)] $f$ is inf-homogeneous and g-comonotone supremal.
\item[\rm(iii)] $f$ is sup-homogeneous and g-comonotone infimal.
\item[\rm(iv)] $f$ is inf-homogeneous and comonotone supremal.
\item[\rm(v)] $f$ is sup-homogeneous and comonotone infimal.
\item[\rm(vi)] $f$ is comonotone supremal and comonotone infimal.
\item[\rm(vii)] $f$ is g-comonotone supremal and g-comonotone infimal.
\item[\rm(viii)] $f$ is Boolean sup-homogeneous and Boolean inf-homogeneous.
\end{enumerate}
\end{theorem}

\begin{proof}
The implications (i) implies (ii) as well as (i) implies (iii) are due to Lemma \ref{lem2}. The converse implications follow from Corollary \ref{lem5} and Remark \ref{rem5}.

The implications (i) implies (iv) and (i) implies (v) are due to Lemma \ref{lem2} and Remark \ref{rem2}, while Corollary \ref{cor2} yields the converse implications.

The equivalence (i) iff (vi) follows from Lemma \ref{lem2}, Remark \ref{rem2} and Corollary \ref{cor1}.

The implication (i) implies (vii) follows from Lemma \ref{lem2}. The converse implication follows from Corollary \ref{cor1}, since if an aggregation function is g-comonotone supremal(infimal) it is also comonotone supremal(infimal), cf. Remark \ref{rem2}.

Finally, (i) iff (viii) is stated in Theorem \ref{thm4}.
\qed
\end{proof}

Let us note that from the computational point of view, to verify any of the two conditions in (vii) one must check $\left|L\right|\cdot 2^n$ pairs of vectors, while to formally check inf(sup)-homogeneity one must consider $\left|L\right|\cdot \left|L\right|^n$ pairs of vectors.

\section{Concluding remarks}\label{sec:5}

We have discussed new axiomatic characterizations of Sugeno integrals on bounded distributive lattices. Some of them are based on new notions of generalized comonotonicity and dual generalized comonotonicity of $L$-valued vectors, and they show new stronger properties satisfied by the Sugeno integrals. On the other hand, some other generalizes and simplifies an axiomatization given in \cite{Couceiro}, bringing as a by-product a drastic reduction of computational complexity of verification an aggregation function to be a Sugeno integral. For example, for a bounded distributive lattice $L$ with cardinality $|L| = k > 2$, the reduction factor is $(k/2)^n$. We expect applications of our results in information fusion and multicriteria decision support when the $L$-valued scales are considered, in particular in the case of linguistic scales different from chains.

Concerning the further research in this area, we aim to study the Sugeno integrals on particular lattices, such as the ordinal or horizontal sums of lattices. Here we expect, among others, to solve a challenging problem of ordinal (horizontal) sums of $L$-valued Sugeno integrals.

\section*{Acknowledgements}
The first author was supported by the international project Austrian Science Fund (FWF)-Grant Agency of the Czech Republic (GA\v{C}R) number 15-34697L; the second author by the Slovak VEGA Grant 1/0420/15; the third author by the project of Palack\'y University Olomouc IGA PrF2015010 and by the Slovak VEGA Grant 2/0044/16.


\section*{References}

\begin{thebibliography}{}


\bibitem{Couceiro}
Couceiro, M., Marichal, J.-L.: Characterizations of discrete Sugeno integrals as polynomial functions over distributive lattices, Fuzzy Sets and Systems 161, 694--707 (2010).

\bibitem{Couceiro2}
Couceiro, M., Marichal, J.-L.: Representations and Characterizations of Polynomial Functions on Chains, Journal of Multiple-Valued Logic \& Soft Computing 16, 65--86 (2010).



\bibitem{IJAR1}
Dubois, D., Prade, H., Rico, A.: Representing qualitative capacities as families of possibility measures, International Journal of Approximate Reasoning 58, 3--24 (2015).
  
\bibitem{IJAR4}
Girotto, B., Holzer, S.: A Chebyshev type inequality for Sugeno integral and comonotonicity, International Journal of Approximate Reasoning 52(3), 444--448 (2011). 

\bibitem{Goguen1967}
Goguen, J.A.: $L$-fuzzy sets, Journal of Mathematical Analysis and Applications 18(1), 145--174 (1967).

\bibitem{Grabisch et al 2009}
Grabisch, M., Marichal, J.-L., Mesiar, R., Pap, E.: Aggregation Functions, Cambridge University Press, Cambridge, 2009.

\bibitem{G1}
Gr\"atzer, G., Lattice Theory: Foundation, Birkh\"auser, Basel, 2011.

\bibitem{HP1}
Hala\v{s}, R., P\'{o}cs, J.: On the clone of aggregation functions on bounded lattices,    
Information Sciences 329, 381--389 (2016).

\bibitem{HP2}
Hala\v{s}, R., P\'{o}cs, J.: On lattices with a smallest set of aggregation functions,    
Information Sciences 325, 316--323 (2015).

\bibitem{HMP}
Hala\v{s}, R., Mesiar, R., P\'ocs, J.: A new characterization of the discrete Sugeno integral, Information Fusion 29, 84--86 (2016).

 

\bibitem{Mac}
Mac Lane, S.: Categories for the Working Mathematician. Second edition, Springer-Verlag, New York, 1998.

\bibitem{Marichal}
Marichal, J.-L.: Weighted lattice polynomials, Discrete Mathematics 309(4), 814--820 (2009).

\bibitem{MAR}
Marichal, J.-L.: An axiomatic approach of the discrete Sugeno integral as a
tool to aggregate interacting criteria in a qualitative framework, IEEE
Transactions on Fuzzy Systems 9(1), 164--172 (2001).
 
\bibitem{IJAR3}
Miranda, E., Montes, I.: Coherent updating of non-additive measures, International Journal of Approximate Reasoning 56, 159--177 (2015).

\bibitem{nove1}
Narukawa, Y., Murofushi, T.: Choquet integral and Sugeno integral as aggregation functions, in Information Fusion in Data Mining (ed. Torra V.), Studies in Fuzziness and Soft Computing 123, 27--39 (2003).


\bibitem{nove2}
Ralescu, D.A., Sugeno,M.: Fuzzy integral representation, Fuzzy Sets and Systems 84(2), 127--133 (1996).

\bibitem{IJAR2}
R\'ebill\'e, Y.: Integral representation of belief measures on compact spaces, International Journal of Approximate Reasoning 60, 37--56 (2015).

\bibitem{SAM}
Saminger-Platz, S., Mesiar, R., Dubois, D.: Aggregation Operators and Commuting.
IEEE Transactions on Fuzzy Systems 15(6), 1032–-1045 (2007).

\bibitem{Sug74}
Sugeno, M.: Theory of fuzzy integrals and its applications, Ph.D. Thesis, Tokyo Institute of Technology, Tokyo, 1974.

\bibitem{Zadeh1965}
Zadeh, L.A.: Fuzzy sets, Information and Control 8(3), 338--353 (1965). 


\end{thebibliography}
\end{document}